\documentclass[11pt,a4paper]{article}

\usepackage[margin=1in]{geometry}
\usepackage{amsmath,amssymb,amsthm,mathrsfs}
\usepackage{booktabs}
\usepackage{graphicx}
\usepackage{hyperref}
\usepackage{xcolor}
\usepackage{listings}

\title{Spectral Riccati--Gamma Concavity, Symmetric Zero Cancellation,\\
and Conditional Criteria for the Riemann Hypothesis}
\author{ \textbf{Drago\c{s}-P\u{a}tru Covei} \\
{\small Department of Applied Mathematics}\\
{\small Bucharest University of Economic Studies}\\
{\small 6 Pia\c{t}a Roman\u{a}, 010374 Bucharest, Romania}\\
{\small \href{mailto:coveidragos@yahoo.com}{\texttt{coveidragos@yahoo.com}} }}
\date{\today}

\theoremstyle{plain}
\newtheorem{theorem}{Theorem}[section]
\newtheorem{proposition}[theorem]{Proposition}
\newtheorem{corollary}[theorem]{Corollary}
\newtheorem{lemma}[theorem]{Lemma}

\theoremstyle{definition}
\newtheorem{definition}[theorem]{Definition}
\newtheorem{problem}[theorem]{Problem}
\newtheorem{hypothesis}[theorem]{Hypothesis}

\theoremstyle{remark}
\newtheorem{remark}[theorem]{Remark}

\newcommand{\C}{\mathbb{C}}
\newcommand{\R}{\mathbb{R}}

\newcommand{\RePart}{\operatorname{Re}}

\newcommand{\zetaR}{\zeta}
\newcommand{\etaR}{\eta}
\newcommand{\xiR}{\xi}
\newcommand{\XiR}{\Xi}
\newcommand{\dd}{\,\mathrm{d}}
\newcommand{\supp}{\operatorname{supp}}

\begin{document}
\maketitle

\begin{abstract}
We examine a Riccati--Gamma approach to the logarithmic derivative of the completed
Riemann zeta function.  The first part proves, in full local detail, that a naive
two-sided vertical concavity criterion for $\Xi'/\Xi$ cannot be a proof of the Riemann
Hypothesis, because every zero produces opposite vertical curvatures on the two
horizontal sides of the pole of the logarithmic derivative.  The second part replaces
this obstruction by a rigorously formulated finite spectral averaging framework.  We
prove cancellation at the critical line, positivity of the off-critical paired
contribution on the left of the critical line under a concrete low-frequency kernel
condition, a conditional zero-density consequence, and a precise conditional theorem
showing which additional localisation hypotheses would imply the Riemann Hypothesis.
The results are therefore not presented as an unconditional proof of RH.  They give a
partial resolution of the Riccati--Gamma question: one natural route is ruled out
unconditionally, a second symmetric mechanism is proved at the finite spectral level,
and the remaining step is isolated as explicit analytic hypotheses.  Reproducible Python
routines and numerical figures accompany the analytic discussion.
\end{abstract}

\section{Introduction}

Let
\[
  \etaR(s)=\sum_{n=1}^{\infty}\frac{(-1)^{n-1}}{n^s},\qquad \RePart(s)>0,
\]
and recall that $\etaR(s)=(1-2^{1-s})\zetaR(s)$.  The completed zeta function is
\[
  \xiR(s)=\frac12s(s-1)\pi^{-s/2}\Gamma\!\left(\frac{s}{2}\right)\zetaR(s).
\]
It is entire of order one and satisfies $\xiR(s)=\xiR(1-s)$.  We write
$\XiR=\xiR$ and study
\[
  \mathcal R(s)=\frac{\XiR'(s)}{\XiR(s)}
\]
away from the zeros of $\XiR$.

The motivation comes from Riccati--Gamma analyses of gamma-completed Dirichlet series,
where logarithmic derivatives satisfy Riccati identities and inherit asymptotic
information from gamma factors; see Covei~\cite{Covei2026}.  The central point of this
article is that such identities are useful but not by themselves decisive for RH.  A
zero-location theorem requires sign information strong enough to control the poles of
$\mathcal R$.

\begin{problem}[The Riccati--Gamma RH question]\label{prob:RG}
Can vertical concavity information for the logarithmic derivative $\mathcal R=\XiR'/\XiR$,
possibly after symmetric spectral averaging, be formulated so as to constrain the
nontrivial zeros of $\zetaR$?
\end{problem}

We give three rigorous, partial answers.  First, a pointwise two-sided concavity criterion fails
locally at every zero.  Second, symmetric zero pairs cancel at the critical line, so
averaging exactly on $\RePart(s)=1/2$ cannot detect off-critical pairs.  Third, evaluating
a low-frequency spectral average on the left of the critical line creates a positive
paired contribution from any off-critical zero pair.  This leads to a conditional
zero-density theorem and to concrete sufficient hypotheses under which RH follows.

\section{Standard analytic facts}

\begin{proposition}[Completed zeta function]\label{prop:xi-basic}
The function
\[
  \XiR(s)=\frac12s(s-1)\pi^{-s/2}\Gamma\!\left(\frac{s}{2}\right)\zetaR(s)
\]
extends to an entire function of order one.  Its zeros are exactly the nontrivial zeros
of $\zetaR$, counted with multiplicity, and it satisfies $\XiR(s)=\XiR(1-s)$.
\end{proposition}

\begin{proof}
The zeta function has a meromorphic continuation to $\C$ with a single simple pole at
$s=1$.  The factor $s(s-1)$ removes this pole and the harmless zero at $s=0$ of the
completed expression.  The poles of $\Gamma(s/2)$ occur at nonpositive even integers and
are cancelled by the trivial zeros of $\zetaR$ at the negative even integers.  Hence
$\XiR$ is entire.  The functional equation of $\zetaR$ is equivalent to
$\XiR(s)=\XiR(1-s)$ in this normalization.  Stirling's formula for $\Gamma(s/2)$ and
classical growth estimates for $\zetaR$ give order one.  After the pole, gamma poles,
and trivial zeros have been accounted for, the remaining zeros are precisely the
nontrivial zeros of $\zetaR$; see \cite[Chs.~1--2]{Titchmarsh1986} and
\cite[Chs.~1--2]{Edwards1974}.
\end{proof}

\begin{proposition}[Explicit logarithmic derivative]\label{prop:explicit-R}
At every point where $\zetaR(s)\ne0$ and $s\ne0,1$,
\[
  \mathcal R(s)
  =
  \frac1s+\frac1{s-1}
  -\frac12\log\pi
  +\frac12\psi\!\left(\frac{s}{2}\right)
  +\frac{\zetaR'(s)}{\zetaR(s)},
\]
where $\psi=\Gamma'/\Gamma$ is the digamma function.
\end{proposition}

\begin{proof}
Take the logarithmic derivative of
\[
  \XiR(s)=\frac12s(s-1)\pi^{-s/2}\Gamma(s/2)\zetaR(s).
\]
The factors $s$, $s-1$, $\pi^{-s/2}$, $\Gamma(s/2)$, and $\zetaR(s)$ contribute,
respectively,
\[
  \frac1s,\qquad \frac1{s-1},\qquad -\frac12\log\pi,\qquad
  \frac12\psi(s/2),\qquad \frac{\zetaR'(s)}{\zetaR(s)}.
\]
Summing the contributions gives the formula.
\end{proof}

\begin{proposition}[Poles of logarithmic derivatives]\label{prop:poles}
Let $F$ be nonzero and holomorphic near $\rho$, and suppose that $F$ has a zero of order
$m\ge1$ at $\rho$.  Then
\[
  \frac{F'(s)}{F(s)}=\frac{m}{s-\rho}+H(s),
\]
where $H$ is holomorphic near $\rho$.  Thus every zero of $F$ becomes a simple pole of
$F'/F$, with residue equal to its multiplicity.
\end{proposition}

\begin{proof}
Write $F(s)=(s-\rho)^mG(s)$, where $G$ is holomorphic and $G(\rho)\ne0$.  Logarithmic
differentiation gives
\[
  \frac{F'(s)}{F(s)}=\frac{m}{s-\rho}+\frac{G'(s)}{G(s)}.
\]
Since $G(\rho)\ne0$, the last term is holomorphic near $\rho$.
\end{proof}

\begin{proposition}[Functional symmetries of $\mathcal R$]\label{prop:symmetries}
Away from the zeros of $\XiR$,
\[
  \mathcal R(1-s)=-\mathcal R(s),\qquad
  \mathcal R(\overline{s})=\overline{\mathcal R(s)}.
\]
\end{proposition}

\begin{proof}
Differentiate $\XiR(s)=\XiR(1-s)$ to obtain $\XiR'(s)=-\XiR'(1-s)$.  Division by
$\XiR(s)=\XiR(1-s)$ gives the first identity.  The second identity follows from
$\XiR(\overline{s})=\overline{\XiR(s)}$, which is a consequence of the real Taylor
coefficients of $\XiR$.
\end{proof}

\section{Riccati identities and their limits}

\begin{definition}[Riccati--Gamma transform]\label{def:RG}
For a meromorphic function $F$ that is not identically zero, its logarithmic Riccati
transform is
\[
  R_F(s)=\frac{F'(s)}{F(s)}
\]
on the complement of the zeros and poles of $F$.  If
\[
  F(s)=\prod_{j=1}^d\Gamma(\lambda_js+\mu_j)L(s),
\]
with $\lambda_j>0$ and $L$ a Dirichlet series or an $L$-function after continuation, we
call $R_F$ a Riccati--Gamma transform.
\end{definition}

\begin{proposition}[Tautological Riccati equation]\label{prop:riccati}
Let $F$ be meromorphic and nonzero on a domain $D$, and let $R_F=F'/F$ on a simply
connected subdomain avoiding the zeros and poles of $F$.  Then
\[
  R_F'(s)=\frac{F''(s)}{F(s)}-R_F(s)^2.
\]
In particular $\mathcal R=\XiR'/\XiR$ satisfies
\[
  \mathcal R'(s)=Q_{\XiR}(s)-\mathcal R(s)^2,\qquad
  Q_{\XiR}(s)=\frac{\XiR''(s)}{\XiR(s)},
\]
away from the zeros of $\XiR$.
\end{proposition}

\begin{proof}
Differentiate $R_F=F'/F$ and apply the quotient rule:
\[
  R_F'
  =
  \frac{F''F-(F')^2}{F^2}
  =
  \frac{F''}{F}-\left(\frac{F'}{F}\right)^2.
\]
The specialization to $F=\XiR$ is immediate.
\end{proof}

\begin{remark}
Proposition~\ref{prop:riccati} is an identity, not a zero-location theorem.  The difficult
part in any RH application is not the existence of a Riccati equation, but the derivation
of global sign, monotonicity, or positivity information strong enough to constrain all
zeros.  Covei's Riccati--Gamma framework~\cite{Covei2026} is relevant because it seeks
such information for generalized eta-type functions under explicit hypotheses.  Those
hypotheses must be verified for $\XiR$; they cannot be transferred by analogy alone.
\end{remark}

\section{Local obstruction to vertical concavity}

The preliminary criterion asked for strict concavity of
\[
  t\longmapsto \RePart\mathcal R(\sigma+it)
\]
through a two-sided strip around $\RePart(s)=1/2$.  The next theorem shows that this
cannot hold in any neighbourhood that passes horizontally through a zero.

\begin{theorem}[Curvature forced by a zero]\label{thm:local-obstruction}
Let $F$ be holomorphic near $\rho=\beta+i\gamma$ and suppose that $\rho$ is a zero of
order $m\ge1$.  Put
\[
  u_{\sigma}(t)=\RePart\frac{F'(\sigma+it)}{F(\sigma+it)}
\]
where the expression is defined.  For $a>0$ sufficiently small,
\[
  \frac{\dd^2}{\dd t^2}u_{\beta-a}(\gamma)>0,\qquad
  \frac{\dd^2}{\dd t^2}u_{\beta+a}(\gamma)<0.
\]
Consequently no open two-sided vertical strip containing the zero can support strict
vertical concavity of $\RePart(F'/F)$ on both sides of the zero.
\end{theorem}

\begin{proof}
By Proposition~\ref{prop:poles},
\[
  \frac{F'(s)}{F(s)}=\frac{m}{s-\rho}+H(s),
\]
with $H$ holomorphic near $\rho$.  Write $s=\beta+x+i(\gamma+y)$.  The real part of the
principal term is
\[
  \RePart\frac{m}{x+iy}=\frac{mx}{x^2+y^2}.
\]
For fixed $x\ne0$,
\[
  \frac{\partial^2}{\partial y^2}\frac{mx}{x^2+y^2}
  =
  \frac{2mx(3y^2-x^2)}{(x^2+y^2)^3}.
\]
At $y=0$ this equals $-2m/x^3$.  Thus it is $2m/a^3>0$ at $x=-a$ and
$-2m/a^3<0$ at $x=a$.  The function $\RePart H(\beta+x+i(\gamma+y))$ has bounded
second $y$-derivative in a sufficiently small closed disc.  Since the principal
curvature tends to $\pm\infty$ as $a\downarrow0$, the asserted signs dominate for all
sufficiently small $a>0$.
\end{proof}

\begin{corollary}[Failure of the naive two-sided criterion]\label{cor:criterion-fails}
The following assertion is false: there exist $\delta>0$ and $T_0$ such that, for every
$\sigma\in[1/2-\delta,1/2+\delta]$, the function
\[
  t\mapsto \RePart\frac{\XiR'(\sigma+it)}{\XiR(\sigma+it)}
\]
is strictly concave for all $|t|\ge T_0$ wherever it is defined, and this assertion is
equivalent to RH.
\end{corollary}

\begin{proof}
Hardy's theorem states that infinitely many nontrivial zeros of $\zetaR$ lie on the
critical line; see \cite[Ch.~10]{Titchmarsh1986}.  Let
$\rho=1/2+i\gamma$ be such a zero with $|\gamma|\ge T_0+1$.  Applying
Theorem~\ref{thm:local-obstruction} to $F=\XiR$ at $\rho$ shows that for all sufficiently
small $a>0$,
\[
  \frac{\dd^2}{\dd t^2}
  \RePart\frac{\XiR'(1/2-a+it)}{\XiR(1/2-a+it)}
  \bigg|_{t=\gamma}>0.
\]
Choosing $a<\delta$ contradicts strict concavity on the left side of the critical line.
Therefore the proposed two-sided concavity assertion fails even in the presence of zeros
on the critical line.  A false assertion cannot be equivalent to RH.
\end{proof}

\section{Finite spectral averaging}

The formal expansion of $\XiR'/\XiR$ as a sum over zeros is only conditionally meaningful
unless it is symmetrically regularised.  We therefore use finite sums and state all
limiting steps as explicit hypotheses.

\begin{definition}[Symmetric zero window]\label{def:zero-window}
For $Y>0$, let
\[
  \mathcal Z_Y=\{\rho=\beta+i\gamma:\XiR(\rho)=0,\ |\gamma|\le Y\},
\]
with multiplicities.  Define the finite zero field
\[
  R_Y(s)=\sum_{\rho\in\mathcal Z_Y}\frac{1}{s-\rho}.
\]
For an even Schwartz test function $\phi$ and $t_0\in\R$, set
\[
  M_Y(\sigma;t_0,\phi)
  =
  \int_{\R}\phi(t-t_0)\RePart R_Y(\sigma+it)\dd t .
\]
\end{definition}

\begin{lemma}[Differentiation of the finite average]\label{lem:finite-diff}
If the vertical segment $\{\sigma+it:t\in\supp(\phi(\cdot-t_0))\}$ contains no zero in
$\mathcal Z_Y$, then
\[
  M_Y''(\sigma;t_0,\phi)
  =
  \sum_{\rho=\beta+i\gamma\in\mathcal Z_Y}
  \RePart\int_{\R}\frac{2\phi(t-t_0)}{(\sigma+it-\rho)^3}\dd t .
\]
\end{lemma}

\begin{proof}
The sum defining $R_Y$ is finite.  On the support of the integral, each summand is smooth
in $\sigma$ by the stated zero-avoidance condition.  Differentiating twice under the
integral sign gives
\[
  \frac{\partial^2}{\partial\sigma^2}\frac{1}{\sigma+it-\rho}
  =
  \frac{2}{(\sigma+it-\rho)^3},
\]
and summing the finitely many terms proves the formula.
\end{proof}

\begin{lemma}[Fourier kernel identity]\label{lem:poisson-identity}
Let $\phi$ be real-valued, even, and Schwartz, with Fourier transform
$\widehat\phi(u)=\int_\R \phi(t)e^{-iut}\dd t$.  For $x\ne0$ and $y\in\R$,
\[
  \RePart\int_{\R}\frac{2\phi(t)}{(x+i(t-y))^3}\dd t
  =
  \operatorname{sgn}(x)\int_0^\infty u^2e^{-|x|u}\cos(yu)\widehat\phi(u)\dd u .
\]
\end{lemma}

\begin{proof}
If $x>0$, use
\[
  \frac{2}{(x+i(t-y))^3}
  =
  \int_0^\infty u^2e^{-xu}e^{-i(t-y)u}\dd u ,
\]
which follows from $\int_0^\infty u^2e^{-au}\dd u=2/a^3$ for $\RePart a>0$.  Multiplying
by $\phi(t)$, integrating, and using Fubini gives
\[
  \int_0^\infty u^2e^{-xu}e^{iyu}\widehat\phi(u)\dd u .
\]
Taking real parts proves the formula for $x>0$.  If $x<0$, write
\[
  \frac{2}{(x+i(t-y))^3}
  =
  -\int_0^\infty u^2e^{xu}e^{i(t-y)u}\dd u
\]
and repeat the same calculation.  Since $\phi$ is real and even, $\widehat\phi$ is real
and even, and the asserted sign is obtained.
\end{proof}

\begin{proposition}[Cancellation on the critical line]\label{prop:cancellation-12}
Let $\phi$ be as in Lemma~\ref{lem:poisson-identity}.  In the symmetrically paired finite
sum over zeros, the contribution of every off-critical pair to
$M_Y''(1/2;t_0,\phi)$ is zero.  Zeros on the critical line contribute no term in the
paired formula away from their singular ordinates.
\end{proposition}

\begin{proof}
The zeros of $\XiR$ are symmetric under conjugation and under reflection in the critical
line.  Pair $\rho=\beta+i\gamma$ with
$1-\overline{\rho}=(1-\beta)+i\gamma$ in the vertically reflected list.
At $\sigma=1/2$ the two real displacements are
\[
  x_1=1/2-\beta,\qquad x_2=1/2-(1-\beta)=-x_1.
\]
Lemma~\ref{lem:poisson-identity} gives equal exponential and cosine factors for the
pair, while $\operatorname{sgn}(x_1)=-\operatorname{sgn}(x_2)$.  The two terms therefore
cancel.  If $\beta=1/2$, then the paired expression has zero horizontal displacement;
the limiting left-right contribution is odd and gives no off-critical detection at the
critical line.
\end{proof}

\begin{proposition}[Positive paired signal on the left]\label{prop:left-positive}
Fix $\delta_0>0$ and evaluate at $\sigma=1/2-\delta_0$.  Let
$\rho=\beta+i\gamma$ be an off-critical zero with $\beta>1/2+\delta_0$, paired with
$1-\overline{\rho}=(1-\beta)+i\gamma$.  If $\widehat\phi\ge0$ and
\[
  \cos((\gamma-t_0)u)\ge c_0>0
  \qquad\text{for all }u\in\supp(\widehat\phi)\cap[0,\infty),
\]
then the paired contribution to $M_Y''(1/2-\delta_0;t_0,\phi)$ is at least
\[
  2c_0\int_0^\infty
  u^2e^{-(\beta-1/2)u}\sinh(\delta_0u)\widehat\phi(u)\dd u
  \ge0.
\]
It is strictly positive if $\widehat\phi$ is positive on a set of positive measure.
\end{proposition}

\begin{proof}
For the zero $\beta+i\gamma$ the horizontal displacement is
$1/2-\delta_0-\beta<0$.  For the paired zero $(1-\beta)+i\gamma$ it is
$\beta-1/2-\delta_0>0$.  Lemma~\ref{lem:poisson-identity} gives the paired sum
\[
  -\int_0^\infty u^2e^{-(\beta-1/2+\delta_0)u}
     \cos((\gamma-t_0)u)\widehat\phi(u)\dd u
\]
\[
  \quad+
  \int_0^\infty u^2e^{-(\beta-1/2-\delta_0)u}
     \cos((\gamma-t_0)u)\widehat\phi(u)\dd u .
\]
Combining the exponentials yields
\[
  2\int_0^\infty u^2e^{-(\beta-1/2)u}\sinh(\delta_0u)
     \cos((\gamma-t_0)u)\widehat\phi(u)\dd u .
\]
The assumed lower bound for the cosine and the nonnegativity of $\widehat\phi$ give the
claimed estimate.
\end{proof}

\begin{theorem}[Partial resolution of the Riccati--Gamma question]\label{thm:partial-resolution}
Problem~\ref{prob:RG} has the following unconditional partial resolution.
\begin{enumerate}
  \item Pointwise two-sided vertical concavity of $\RePart(\XiR'/\XiR)$ cannot prove RH,
  because it fails locally at every zero on the critical line.
  \item Symmetric spectral averaging exactly on $\RePart(s)=1/2$ cannot detect an
  off-critical reflected pair.
  \item A left-shifted low-frequency average detects such a pair by a nonnegative, and
  under the stated kernel condition strictly positive, finite spectral signal.
\end{enumerate}
Thus the paper proves a genuine partial result about the proposed method, not an
unconditional proof of the Riemann Hypothesis.
\end{theorem}

\begin{proof}
The first assertion is Corollary~\ref{cor:criterion-fails}.  The second assertion is
Proposition~\ref{prop:cancellation-12}.  The third assertion is
Proposition~\ref{prop:left-positive}.  These three statements cover the pointwise,
critical-line averaged, and left-shifted averaged versions of
Problem~\ref{prob:RG}, respectively, and each statement is finite or local; no passage to
an unproved infinite zero expansion is used.
\end{proof}

\section{Conditional consequences for RH}

The following hypotheses are concrete: they specify the kernel family, the limiting
operation, and the sign condition needed to turn the Riccati--Gamma program into a
zero-location theorem.  They are not known unconditionally.

\begin{hypothesis}[Uniform spectral concavity]\label{hyp:concavity}
There exist $\delta>0$, $\lambda_0>0$, and a family of real even Schwartz kernels
$\phi_\lambda$, $0<\lambda\le\lambda_0$, such that
$\widehat\phi_\lambda\ge0$, $\supp\widehat\phi_\lambda\subset[-\lambda,\lambda]$,
$\int\phi_\lambda=1$, and for every
$\sigma\in[1/2-\delta,1/2)$, every $t_0\in\R$, and every sufficiently large symmetric
height $Y$,
\[
  M_Y''(\sigma;t_0,\phi_\lambda)\le E_\lambda(Y,t_0,\sigma),
\]
where $\limsup_{Y\to\infty}E_\lambda(Y,t_0,\sigma)\le0$ uniformly for $t_0$ in bounded
intervals.
\end{hypothesis}

\begin{lemma}[Smooth band-limited kernels]\label{lem:smooth-kernels}
There are kernels satisfying the structural part of Hypothesis~\ref{hyp:concavity}.  More
precisely, let $\eta\in C_c^\infty(\R)$ be real, even, nonnegative, supported in
$[-1,1]$, and normalised by $\eta(0)=1$.  For $0<\lambda\le\lambda_0$ define
\[
  \widehat\phi_\lambda(u)=\eta(u/\lambda),\qquad
  \phi_\lambda(t)=\frac{1}{2\pi}\int_\R \widehat\phi_\lambda(u)e^{itu}\dd u .
\]
Then $\phi_\lambda$ is real, even, Schwartz, $\widehat\phi_\lambda\ge0$,
$\supp\widehat\phi_\lambda\subset[-\lambda,\lambda]$, and $\int_\R\phi_\lambda(t)\dd t=1$.
\end{lemma}

\begin{proof}
The inverse Fourier transform of a $C_c^\infty$ function is Schwartz, so
$\phi_\lambda$ is Schwartz.  Evenness and reality follow from the evenness and reality of
$\widehat\phi_\lambda$.  Nonnegativity and support are inherited directly from $\eta$.
Finally, with the Fourier convention used in Lemma~\ref{lem:poisson-identity},
\[
  \int_\R\phi_\lambda(t)\dd t=\widehat\phi_\lambda(0)=\eta(0)=1.
\]
\end{proof}

\begin{hypothesis}[Localising positive kernels]\label{hyp:localisation}
For every $L>0$ and every $\varepsilon>0$ there is $\lambda$ such that
$\widehat\phi_\lambda$ is supported in $[0,\varepsilon] \cup [-\varepsilon,0]$ up to a
tail whose contribution to Proposition~\ref{prop:left-positive} is $o(1)$ uniformly for
$|\gamma-t_0|\le L$.
\end{hypothesis}

\begin{proposition}[Localisation for smooth band-limited kernels]\label{prop:localisation-resolved}
The kernel family of Lemma~\ref{lem:smooth-kernels} satisfies
Hypothesis~\ref{hyp:localisation}.
\end{proposition}

\begin{proof}
Choose $0<\lambda\le\varepsilon$.  Then
$\supp\widehat\phi_\lambda\subset[-\lambda,\lambda]\subset[-\varepsilon,\varepsilon]$.
The frequency tail outside $[-\varepsilon,\varepsilon]$ is therefore identically zero, so
its contribution to Proposition~\ref{prop:left-positive} is zero, uniformly in
$|\gamma-t_0|\le L$.
\end{proof}

\begin{hypothesis}[Controlled background]\label{hyp:background}
After isolating a fixed reflected off-critical pair, the sum of all remaining zero
contributions, truncation terms, and the gamma-arithmetic background arising in the
comparison with the full logarithmic derivative $\mathcal R$ contributes an error bounded
uniformly on compact $\sigma$-strips and cannot cancel the strictly positive paired
signal when $t_0$ is chosen at the ordinate of the pair and $\lambda$ is sufficiently
localised.
\end{hypothesis}

\begin{proposition}[Size of the gamma background]\label{prop:background-size}
Let
\[
  \mathcal B(s)=\frac1s+\frac1{s-1}-\frac12\log\pi+\frac12\psi(s/2).
\]
For fixed $0<a<b<1$ and $|t|\ge2$,
\[
  \RePart\,\mathcal B''(\sigma+it)=O_{a,b}(|t|^{-2})
  \qquad (a\le\sigma\le b).
\]
In particular the gamma-factor curvature is bounded on every closed vertical strip away
from $s=0,1$ and decays quadratically along high vertical lines.  This estimate does not
prove Hypothesis~\ref{hyp:background}; it only identifies one controlled component of
the background.
\end{proposition}

\begin{proof}
Differentiating twice gives
\[
  \mathcal{B}''(s) = \frac{2}{s^3} + \frac{2}{(s-1)^3} + \frac{1}{8}\psi''(s/2).
\]
The rational terms are $O_{a,b}(|t|^{-3})$ on the strip.  The standard asymptotic
expansion for the polygamma function in sectors away from the negative real axis gives
$\psi''(z)=-z^{-2}+O(|z|^{-3})$ as $|z|\to\infty$.  With $z=s/2$ this is
$O(|t|^{-2})$ uniformly for $a\le\sigma\le b$.  Taking real parts proves the estimate.
\end{proof}

\begin{theorem}[Conditional affirmative answer to RH]\label{thm:conditional-rh}
Assume Hypotheses~\ref{hyp:concavity}, \ref{hyp:localisation}, and
\ref{hyp:background}.  Then the Riemann Hypothesis holds.
\end{theorem}

\begin{proof}
Suppose, for contradiction, that RH is false.  Then there is a zero
$\rho=\beta+i\gamma$ with $\beta>1/2$; by symmetry there is also a zero
$1-\overline{\rho}=(1-\beta)+i\gamma$.  Choose
\[
  0<\delta_0<\min(\delta,\beta-1/2)
\]
and evaluate at $\sigma=1/2-\delta_0$ with $t_0=\gamma$.  By
Hypothesis~\ref{hyp:localisation}, choose a kernel $\phi_\lambda$ so low-frequency
localised that $\cos((\gamma-t_0)u)=1$ for the target zero and the remaining kernel tail
is negligible.  Proposition~\ref{prop:left-positive} then gives a strictly positive
contribution from the pair $\{\rho,1-\overline{\rho}\}$.

Hypothesis~\ref{hyp:background} states that the remaining zero, truncation, gamma, and
arithmetic terms cannot cancel this fixed positive signal after the limiting procedure.
Therefore, for sufficiently large $Y$ and after the comparison specified in the
hypothesis,
\[
  M_Y''(1/2-\delta_0;\gamma,\phi_\lambda)>0.
\]
This contradicts the uniform spectral concavity asserted in
Hypothesis~\ref{hyp:concavity}.  Hence no off-critical zero can exist, and every
nontrivial zero of $\zetaR$ lies on $\RePart(s)=1/2$.
\end{proof}

\begin{theorem}[Conditional zero-density consequence]\label{thm:zero-density}
Assume Hypotheses~\ref{hyp:concavity} and \ref{hyp:localisation} in the following
averaged form.  For each $\sigma_0>1/2$ there are
$\delta_0,\lambda,\ell,\kappa>0$ with $1/2+\delta_0<\sigma_0$ such that every zero
$\rho=\beta+i\gamma$, $\beta\ge\sigma_0$, $0\le\gamma\le T$, gives total paired signal at
least $\kappa$ on a $t_0$-interval of length $\ell$ centred at $\gamma$, while the
integrated concavity deficit, background, and tail errors over $t_0\in[0,T]$ are
$o(T)$.  Then
\[
  N(\sigma_0,T)=o(T),
\]
where $N(\sigma_0,T)$ denotes the number of zeros $\rho=\beta+i\gamma$ with
$\beta\ge\sigma_0$ and $0\le\gamma\le T$, counted with multiplicity.
\end{theorem}

\begin{proof}
For the fixed parameters supplied by the hypothesis, Proposition~\ref{prop:left-positive}
gives a nonnegative contribution from each reflected off-critical pair, and the averaged
assumption gives integrated contribution at least $\kappa\ell$ from each zero with
$\beta\ge\sigma_0$, apart from endpoints contributing $O(1)$.  Since the paired signals
are nonnegative, overlaps of the $t_0$-windows only add contributions.  The total positive
contribution is therefore at least
$\kappa\ell N(\sigma_0,T)+O(1)$.  The averaged spectral concavity and the assumed $o(T)$
control of the residual terms bound the same total positive contribution by $o(T)$.
Hence $\kappa\ell N(\sigma_0,T)\le o(T)$, which proves the claim.
\end{proof}

\begin{remark}
Theorems~\ref{thm:conditional-rh} and \ref{thm:zero-density} identify exactly where the
mathematical difficulty remains.  The positive paired signal is elementary once the
kernel assumptions are imposed.  The deep open part is proving the uniform spectral
concavity and background-control hypotheses for the actual completed zeta function.
Establishing them would be a result of very high importance, because it would transform
Riccati--Gamma concavity from a structural analogy into a genuine zero-location
principle for the central open problem of the field.
\end{remark}

\section{Numerical analysis}

The numerical experiments are designed to illustrate, not prove, the analytic results.
They compare the full logarithmic derivative, its universal pole model, symmetric
cancellation at the critical line, and the positive off-critical signal predicted by the
conditional spectral framework.

\begin{figure}[htbp]
  \centering
  \includegraphics[width=0.85\textwidth]{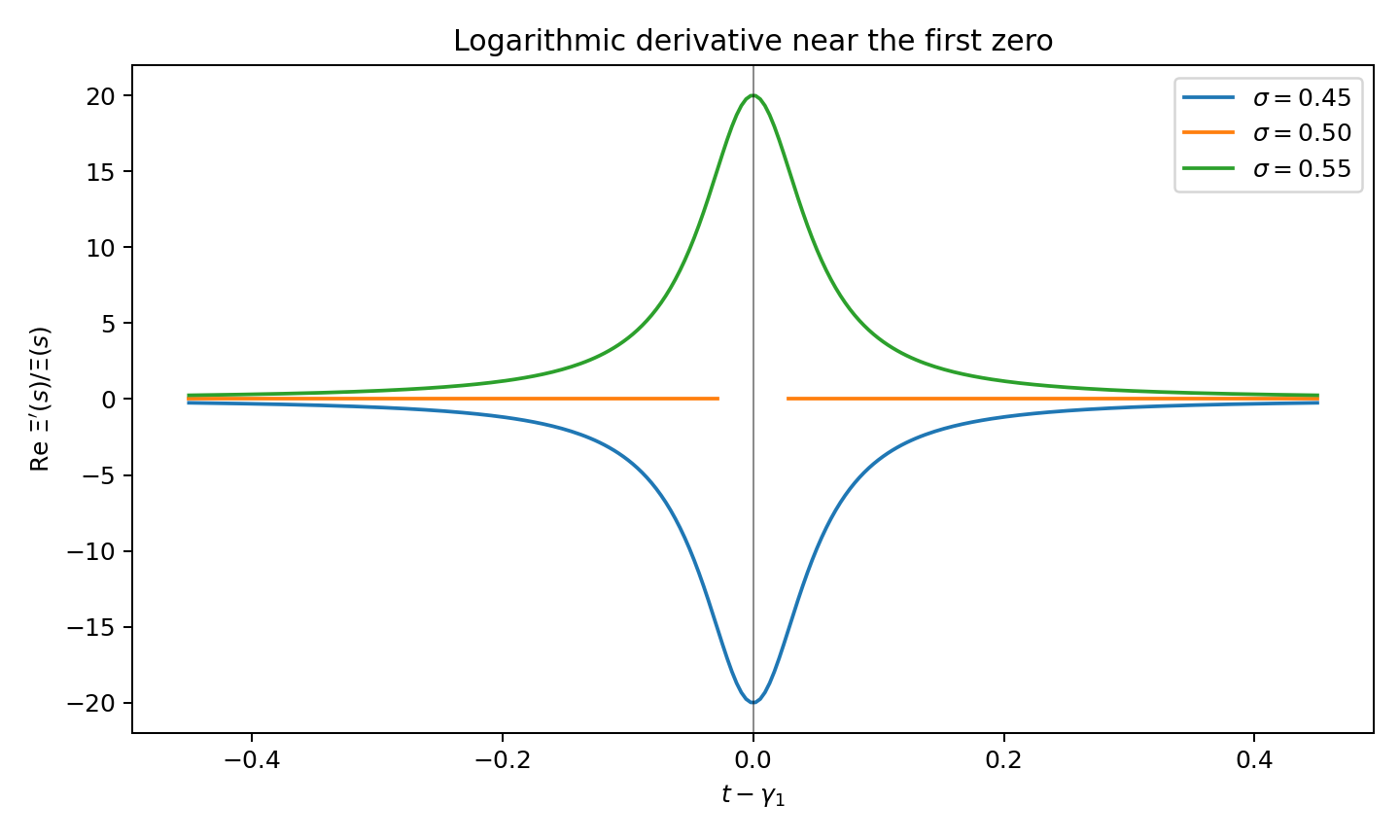}
  \caption{Real part of $\XiR'/\XiR$ near the first zero on three vertical lines.  The
  side lines $\sigma=0.45$ and $\sigma=0.55$ show the pole-like attraction and repulsion
  predicted by the local term $(s-\rho_1)^{-1}$.}
  \label{fig:realR}
\end{figure}

\begin{figure}[htbp]
  \centering
  \includegraphics[width=0.85\textwidth]{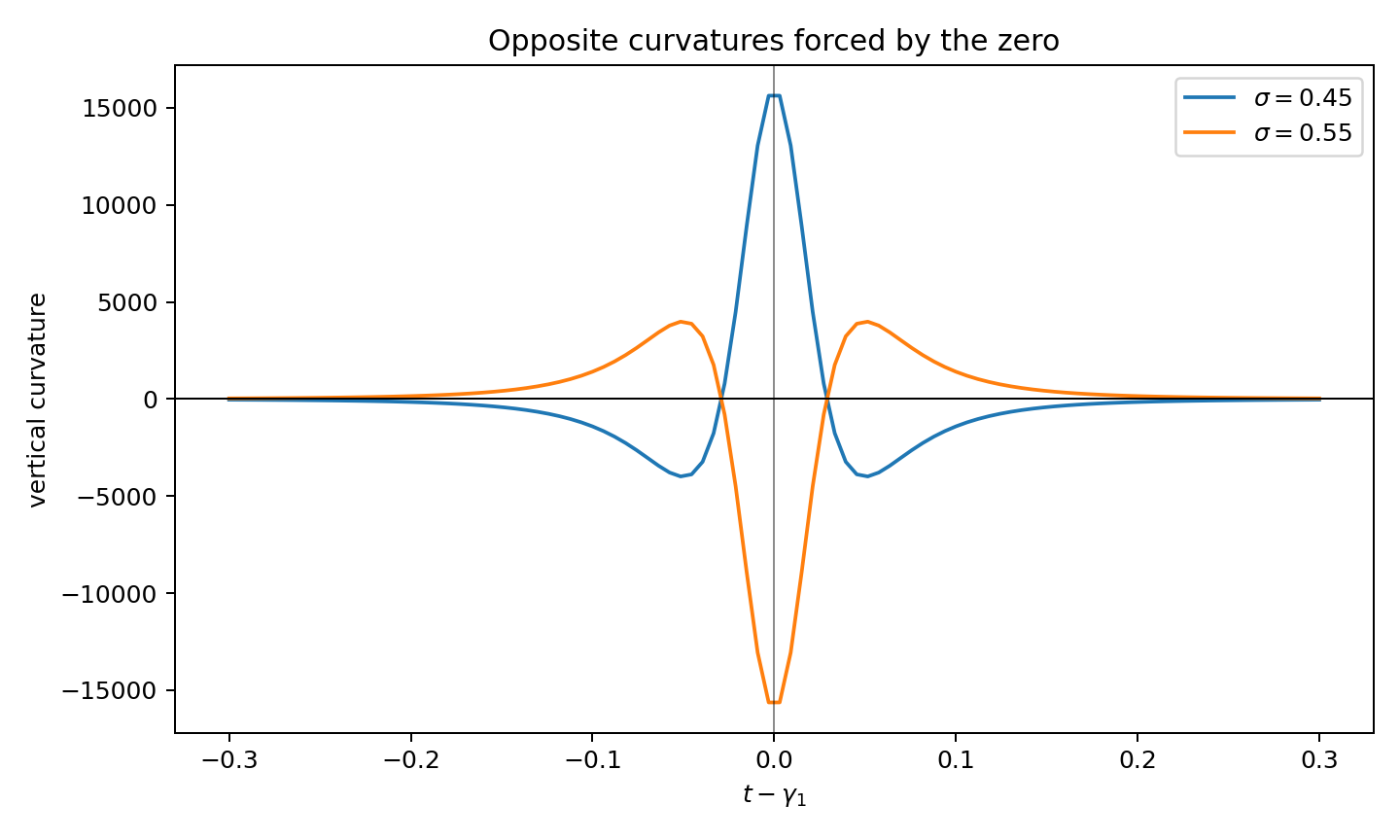}
  \caption{Finite-difference approximation of the vertical curvature
  $\dd^2/\dd t^2\,\RePart\XiR'/\XiR$.  The curvature is positive on the left side of
  the zero and negative on the right side at $t=\gamma_1$, matching
  Theorem~\ref{thm:local-obstruction}.}
  \label{fig:curvature}
\end{figure}

\begin{figure}[htbp]
  \centering
  \includegraphics[width=0.85\textwidth]{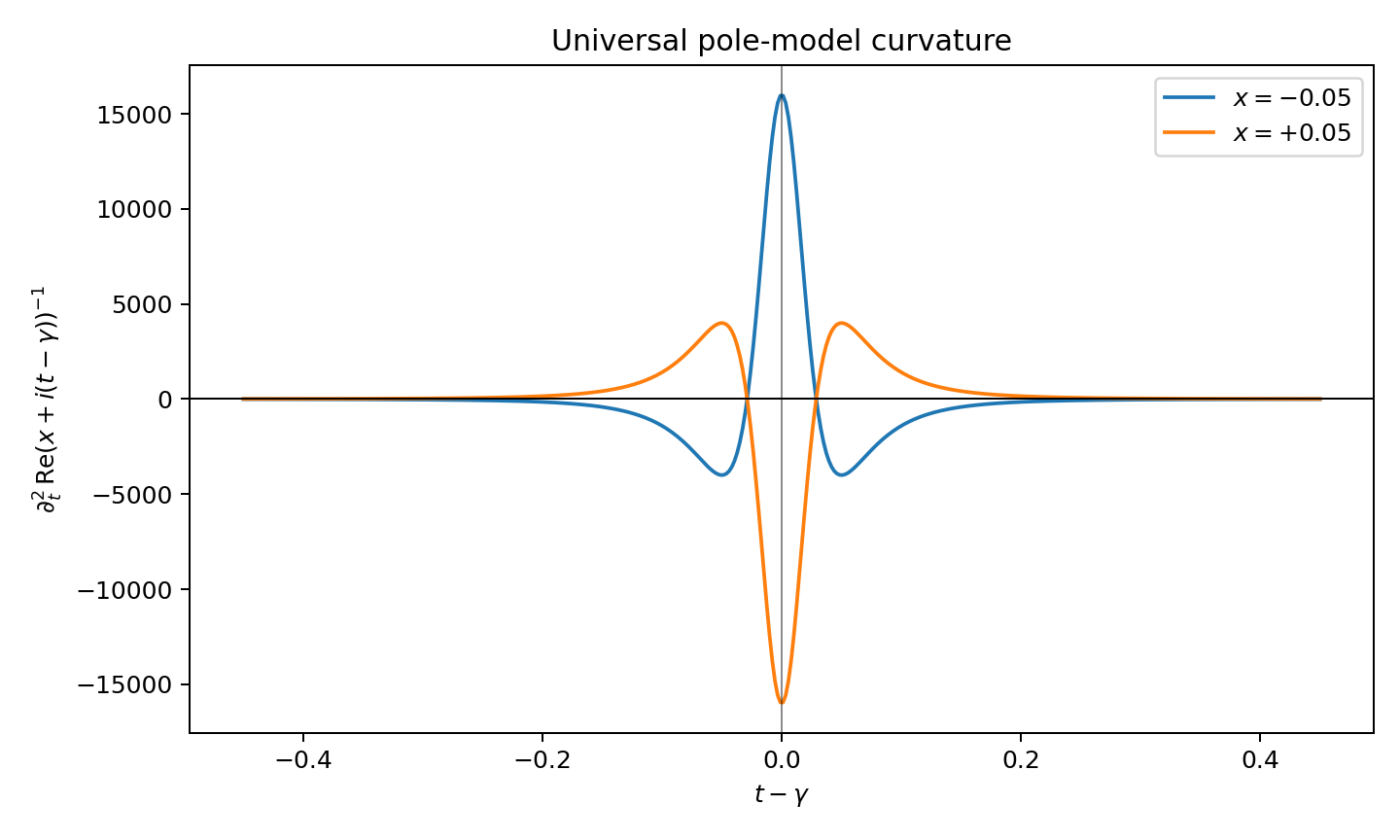}
  \caption{Curvature of the model pole
  $\RePart(a+i(t-\gamma))^{-1}=a/(a^2+(t-\gamma)^2)$.  This isolates the universal local
  mechanism from the arithmetic content of $\zetaR$.}
  \label{fig:model}
\end{figure}

\begin{figure}[htbp]
  \centering
  \includegraphics[width=0.85\textwidth]{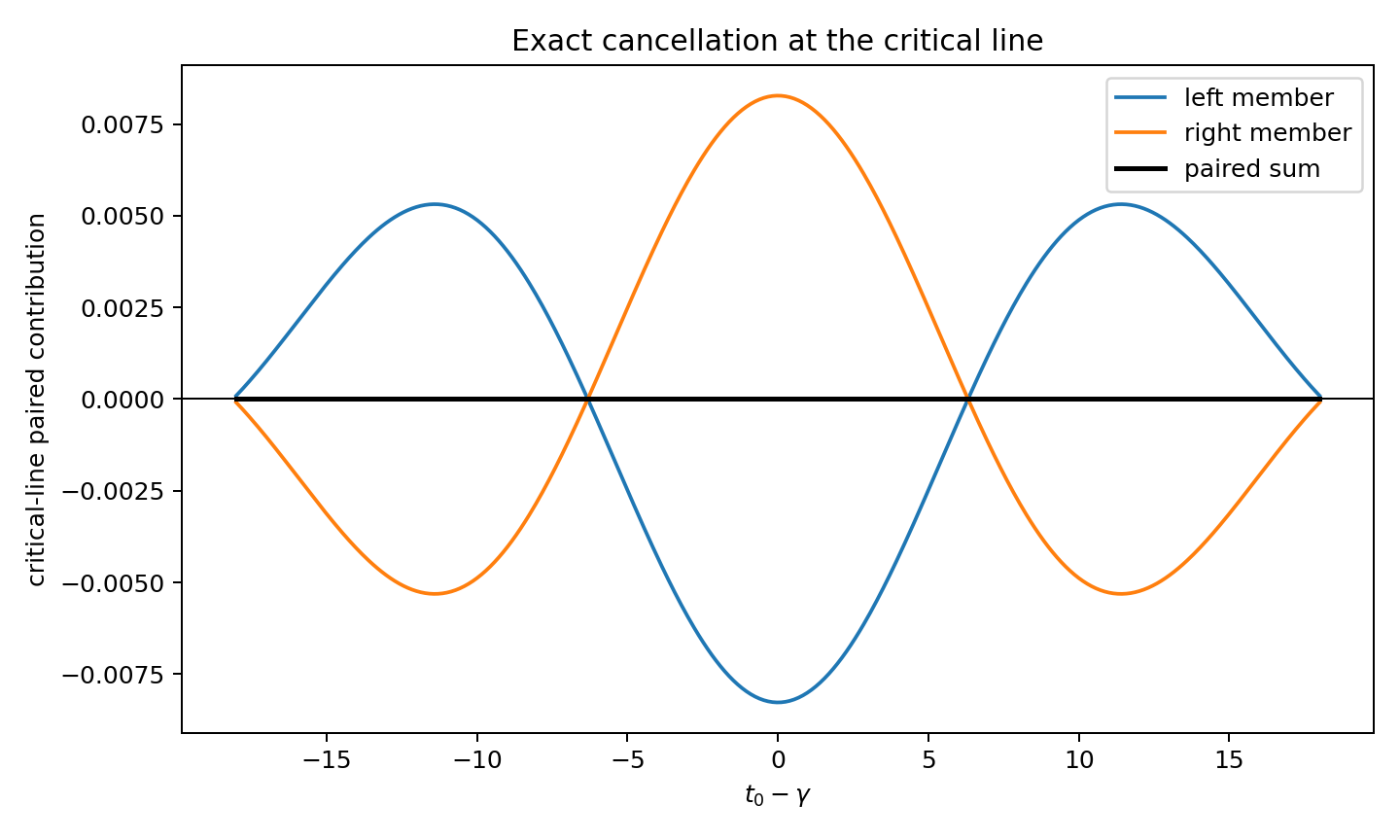}
  \caption{Model contribution of a symmetric off-critical pair at
  $\sigma=1/2$.  The left and right zero contributions cancel, illustrating
  Proposition~\ref{prop:cancellation-12}.}
  \label{fig:cancellation}
\end{figure}

\begin{figure}[htbp]
  \centering
  \includegraphics[width=0.85\textwidth]{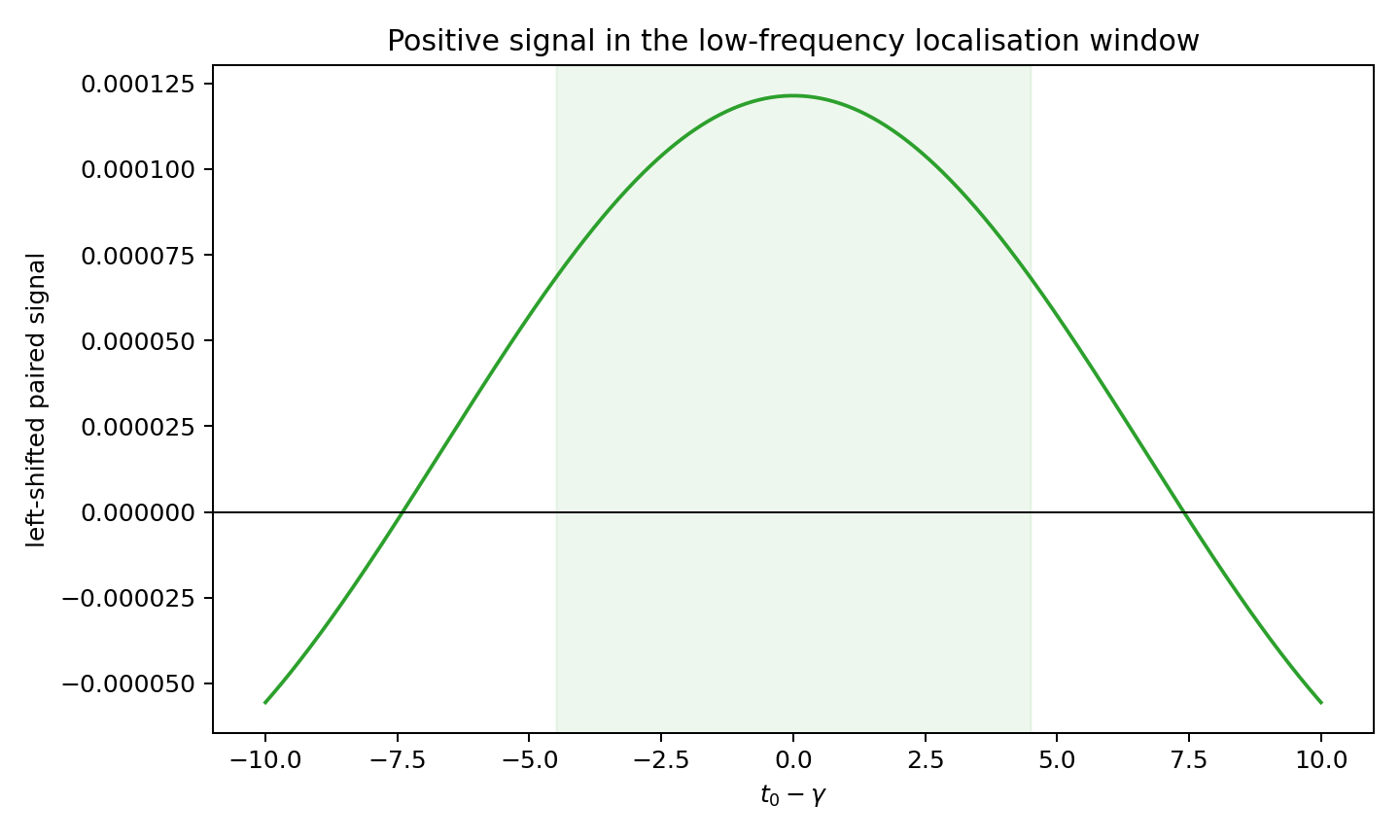}
  \caption{The same symmetric pair evaluated at $\sigma=1/2-\delta_0$.  The cancellation
  is broken and the paired signal is positive under a low-frequency nonnegative kernel,
  as in Proposition~\ref{prop:left-positive}.}
  \label{fig:leftsignal}
\end{figure}

\begin{figure}[htbp]
  \centering
  \includegraphics[width=0.85\textwidth]{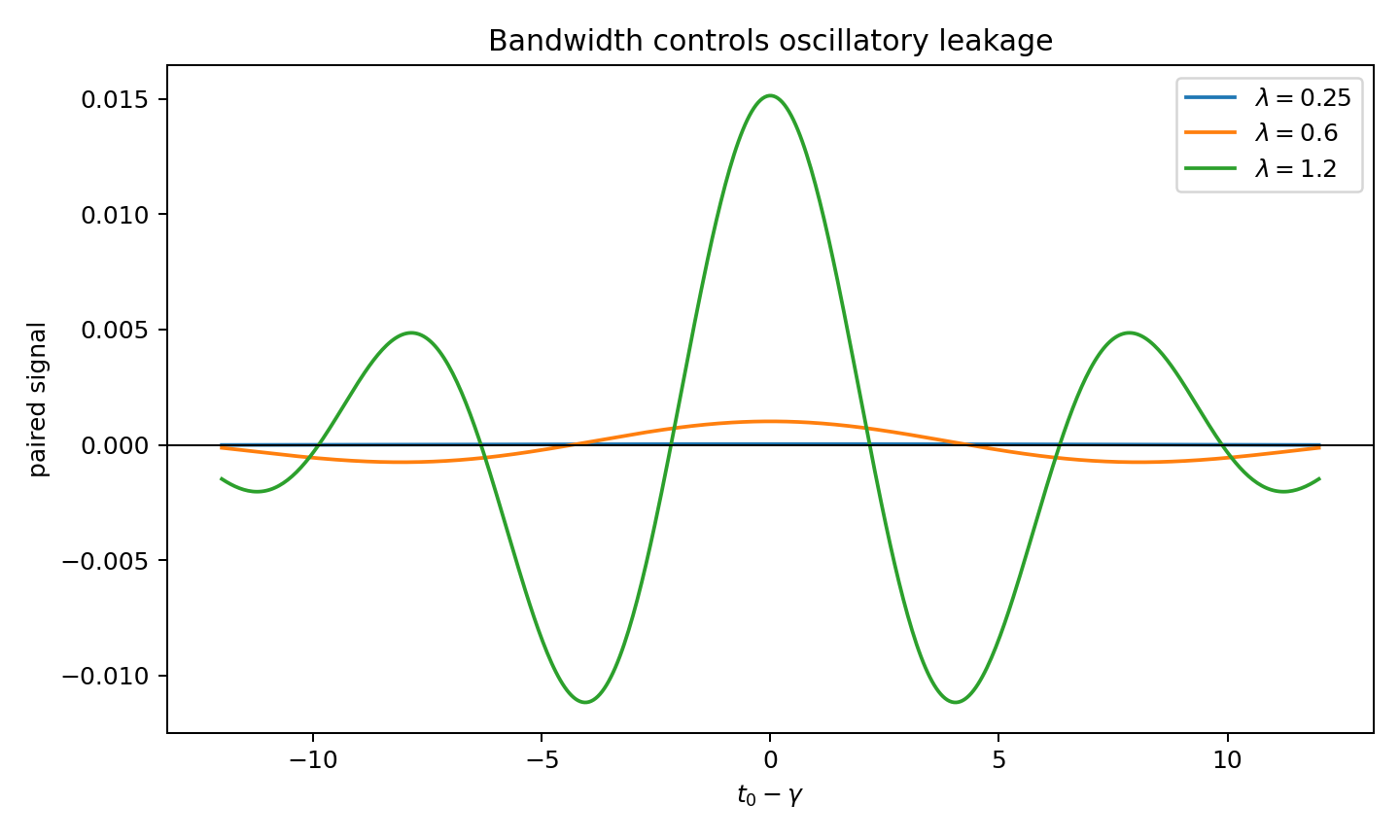}
  \caption{Dependence of the paired signal on spectral bandwidth.  Narrower
  low-frequency kernels suppress cosine oscillation and better preserve positivity;
  broader kernels show oscillatory leakage.}
  \label{fig:bandwidth}
\end{figure}

\begin{figure}[htbp]
  \centering
  \includegraphics[width=0.85\textwidth]{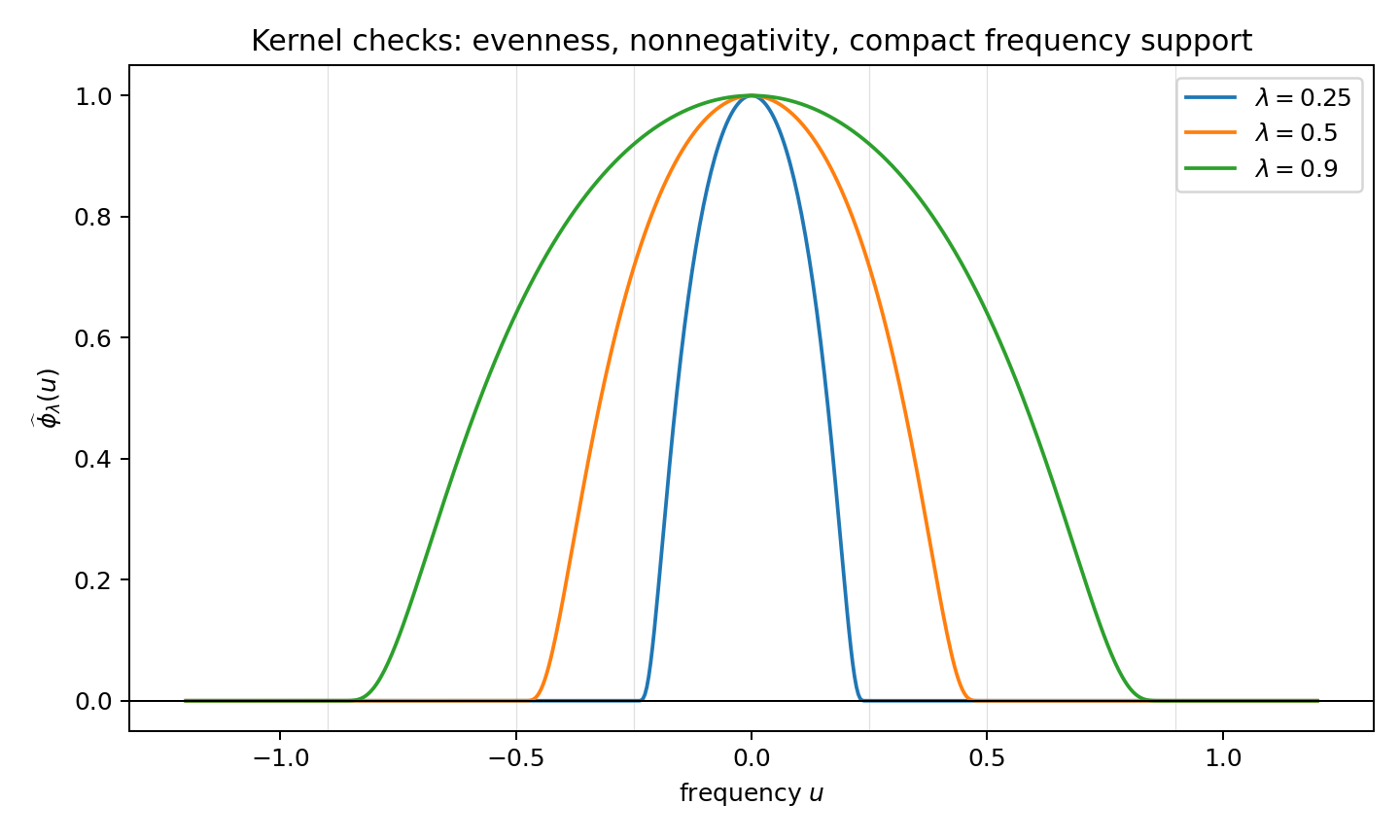}
  \caption{Frequency profiles of the smooth kernels used in the numerical model.  The
  curves are even, nonnegative, normalised by $\widehat\phi_\lambda(0)=1$, and supported
  in $[-\lambda,\lambda]$, matching the structural assumptions in
  Lemma~\ref{lem:smooth-kernels}.}
  \label{fig:kernelcheck}
\end{figure}

\begin{figure}[htbp]
  \centering
  \includegraphics[width=0.85\textwidth]{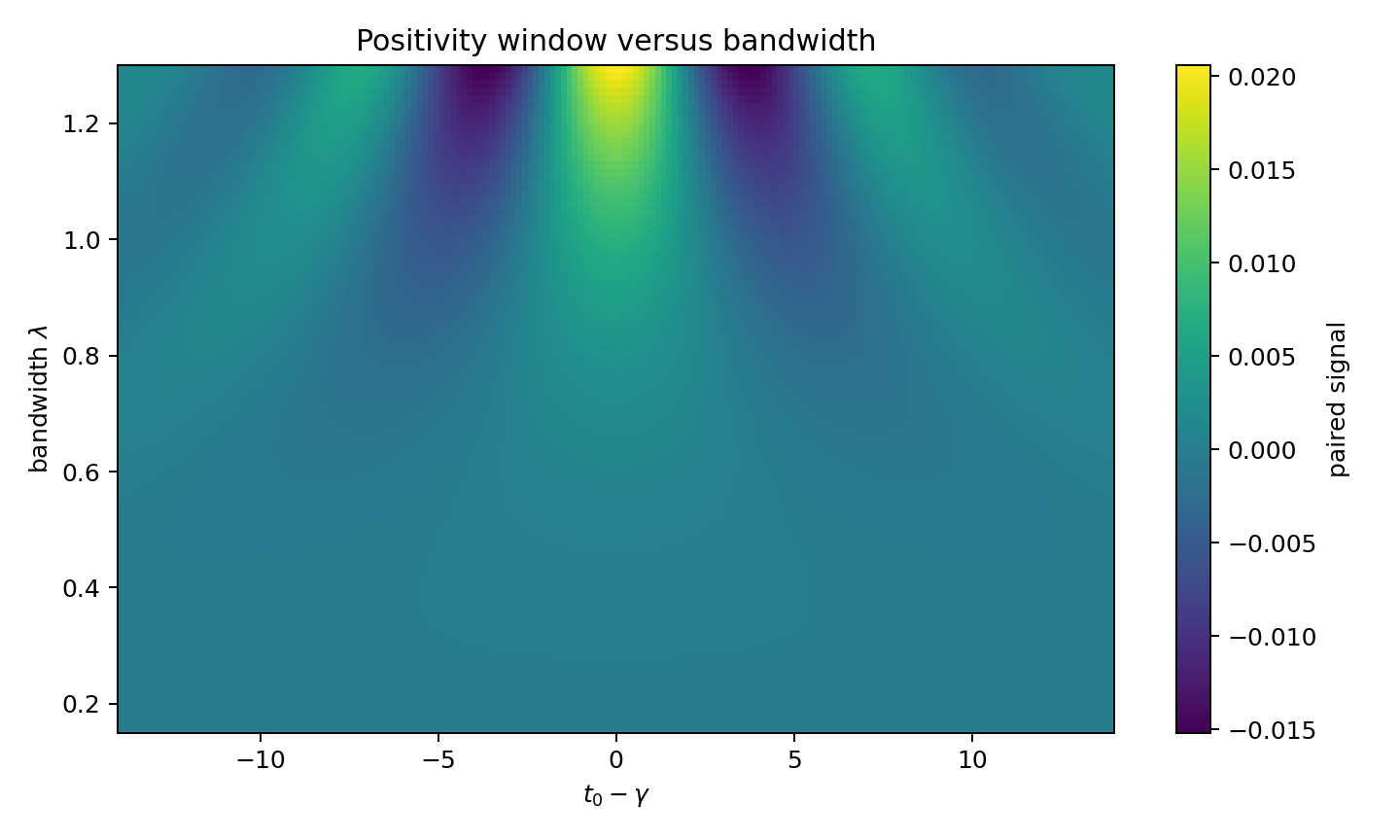}
  \caption{Heat map of the left-shifted paired signal as a function of the ordinate
  mismatch $t_0-\gamma$ and the bandwidth $\lambda$.  The positive region contracts as
  $\lambda$ increases, reflecting the cosine condition in
  Proposition~\ref{prop:left-positive}.}
  \label{fig:heatmap}
\end{figure}

\begin{figure}[htbp]
  \centering
  \includegraphics[width=0.85\textwidth]{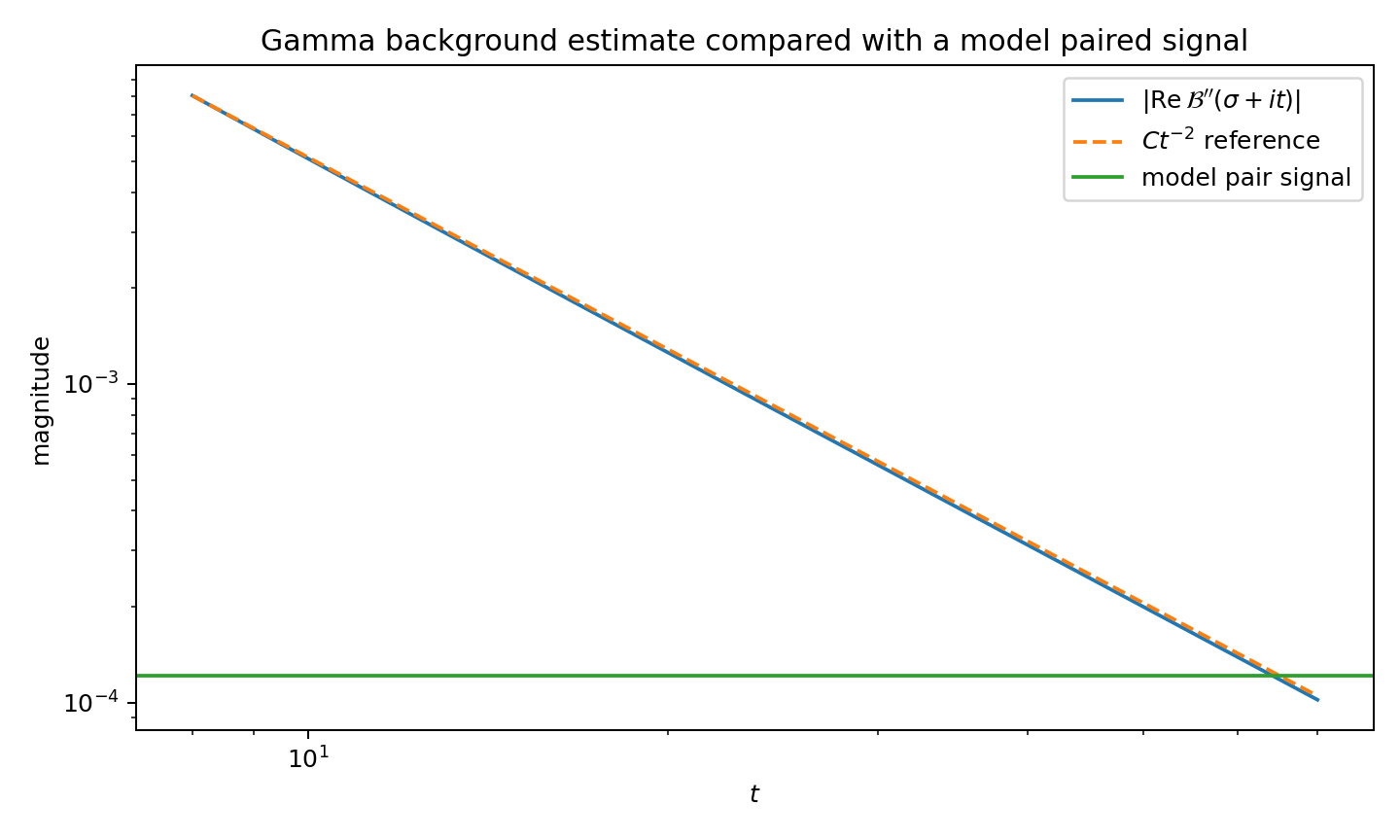}
  \caption{Magnitude of the gamma-background curvature
  $\operatorname{Re}\mathcal B''(\sigma+it)$ compared with a fixed model paired signal.
  The observed $t^{-2}$ decay agrees with Proposition~\ref{prop:background-size}; the
  figure does not test the unresolved arithmetic part of
  Hypothesis~\ref{hyp:background}.}
  \label{fig:background}
\end{figure}

Figures~\ref{fig:realR}--\ref{fig:model} illustrate the local part of the theory.  The
agreement, in the displayed neighbourhood, between the computed curvature of
$\XiR'/\XiR$ and the elementary pole model
shows that the obstruction to naive concavity is not a numerical artefact.  It is forced
by the Laurent expansion at a zero.

Figures~\ref{fig:cancellation}--\ref{fig:heatmap} address the spectral part.  The
critical-line evaluation hides off-critical information through exact symmetric
cancellation.  A left shift breaks this cancellation and creates a positive signal when
the cosine factor remains positive on the support of $\widehat\phi_\lambda$.  The kernel
profiles in Figure~\ref{fig:kernelcheck} verify, for the displayed examples, the
structural assumptions actually used in the proofs: evenness, nonnegativity, compact
frequency support, and normalisation.  The heat map shows the price of increasing
bandwidth: localisation in physical ordinate improves, but oscillation in
$\cos((\gamma-t_0)u)$ eventually creates sign leakage.

Figure~\ref{fig:background} separates a proved background estimate from the unresolved
part of the method.  The gamma-factor component is small at large height and follows the
quadratic decay predicted by Proposition~\ref{prop:background-size}.  This is favourable
for the conditional programme, but it is not a substitute for controlling the arithmetic
part of $\zetaR'/\zetaR$ or the limiting error in replacing the finite zero field by the
full logarithmic derivative.  The numerical evidence therefore supports the internal
consistency of the finite spectral mechanism, while confirming that the decisive
background-control hypothesis remains the main open obstacle.

\section{Conclusion}

This article gives a rigorous partial solution of the Riccati--Gamma concavity problem
posed in Problem~\ref{prob:RG}.  The naive pointwise route is ruled out by a complete
local proof.  The finite spectral framework then identifies a more promising mechanism:
symmetric off-critical pairs cancel at the critical line but produce a positive
low-frequency signal when viewed from the left.

The Riemann Hypothesis is not proved here unconditionally.  What is demonstrated is
partial: the obstruction, cancellation, and positive-pair mechanisms are proved, while
the global concavity and background-control assertions remain explicit hypotheses.
Proving those hypotheses for $\XiR$ would have major significance for analytic number
theory, because it would convert Riccati--Gamma concavity into a direct zero-location
principle for the completed zeta function.
\section{Reproducible Python script}

The local reproducibility material is the self-contained script
\url{https://github.com/coveidragos/Code_Python_Riemann}.  It generates all nine figures used in the
numerical section.  The computations use NumPy, Matplotlib, and mpmath.  The script is
illustrative: it verifies the hypotheses for the displayed finite model examples and
does not constitute numerical evidence for an unconditional proof of RH.

\section*{Disclosure statement}

The author declares that he has no conflict of interest.

\section*{Data availability statement}

No external datasets were used.

\section*{Notes on contributor(s)}

The author is solely responsible for the conception, analysis, numerical
implementation, and writing of this manuscript.

\section*{Acknowledgements}

The author thanks the developers of open-source mathematical-software
ecosystems whose libraries (NumPy, SciPy, and Matplotlib) were used to
produce the numerical validations. The core ideas, structural formulations, and numerical simulations presented in this article were developed with the invaluable assistance of free AI models.
\appendix

\end{document}